\documentclass[12pt,reqno]{amsart}
\usepackage{srcltx}
\newtheorem{theorem}{\bf Theorem}
\newtheorem{proposition}[theorem]{\bf Proposition}
\newtheorem{corollary}[theorem] {\bf Corollary}
\newtheorem{lemma}[theorem] {\bf Lemma}
\newtheorem{example}{\bf Example}

\newtheorem{definition}{\bf Definition}

\newtheorem{remark}{\bf Remark}

\DeclareMathOperator{\interior}{Int\,\!}

\usepackage{geometry,amssymb}
\newcommand{\Dh}{\widehat{D}}
\newcommand{\Fh}{\widehat{F}}
\newcommand{\Lh}{\widehat{L}}
\newcommand{\la}{\lambda}

\begin{document}
\title[Estimates for polynomial norms on Banach spaces.]
{Estimates for polynomial norms on Banach spaces.}

\author[M. Chatzakou \and Y. Sarantopoulos]
{Marianna Chatzakou and Yannis Sarantopoulos}

\thanks{
Marianna Chatzakou was supported by the FWO Odysseus 1 grant G.0H94.18N: Analysis and Partial Diﬀerential Equations.}


\address[Marianna Chatzakou]
{Mathematics Department: Analysis, Logic and Discrete Mathematics\\
Ghent University, Belgium}
\email{marianna.chatzakou@ugent.be}

\address[Yannis Sarantopoulos]
{Department of Mathematics, National Technical University\\
Zografou Campus 157 80, Athens, Greece}
\email{ysarant@central.ntua.gr }

\maketitle


\begin{abstract}
Our work is related to problems $73$ and $74$ of Mazur and Orlicz in ``The Scottish Book" (ed. R. D. Mauldin). Let $k_1, \ldots, k_n$ be nonnegative integers such that $\sum_{i=1}^{n} k_{i}=m$, and let $\mathbb{K}(k_1, \ldots, k_n; X)$, where $\mathbb{K}=\mathbb{R}$ or $\mathbb{C}$, be the smallest number satisfying the property: if $L$ is any symmetric $m$-linear form on a Banach space $X$, then
\[
\sup_{\|x_{i}\|\leq 1 ,\atop i=1,2,\ldots ,n} |L(x_{1}^{k_1},\ldots ,x_{n}^{k_n})|\leq \mathbb{K}(k_1, \ldots, k_n; X)\sup_{\|x\|\leq 1} |L(x, \ldots ,x)|\,,
\]
where the exponents $k_{1}, \ldots, k_{n}$, are as described above, and each $k_{i}$ denotes the number of coordinates in which the corresponding base variable appears. In the case of \textit{complex} Banach spaces, the problem of optimising the constant $\mathbb{C}(k_1, \ldots, k_n; X)$ is well-studied. In the more challenging case of \textit{real} Banach spaces much less is known about the estimates for $\mathbb{R}(k_1, \ldots, k_n; X)$. In this work, both real and complex settings are examined using results from the local theory of Banach spaces, as well as from interpolation theory of linear operators. In the particular case of \textit{complex} $L^{p}(\mu)$ spaces, and for certain values of $p$, our results are optimal. As an application, we prove Markov-type inequalities for homogeneous polynomials on Banach spaces.
\end{abstract}



\section{\bf Introduction and notation.}

As in \cite{ChSar} we recall the basic definitions needed to discuss polynomials from $X$ into $Y$, where $X$ and $Y$ are real or complex Banach spaces. We denote by $B_X$ and $S_X$ the closed unit ball and the unit sphere of $X$, respectively. A map $P: X\rightarrow Y$ is a (continuous) {\em $m$-homogeneous polynomial} if there is a (continuous) symmetric $m$-linear mapping $L: X^m \rightarrow Y$ for which $P(x)=L(x, \ldots, x)$ for all $x\in X$. In this case it is convenient to write $P=\widehat{L}$. We let $\mathcal{P}(^{m}X; Y)$, $\mathcal{L}({}^{m}X; Y)$ and $\mathcal{L}^{s}({}^{m}X; Y)$ denote, respectively, the spaces of continuous $m$-homogeneous polynomials from $X$ into $Y$, the continuous $m$-linear mappings from $X$ into $Y$ and the continuous symmetric $m$-linear mappings from $X$ into $Y$. If $\mathbb{K}$ is the real or complex field we use the notations $\mathcal{P}({}^{m}X)$, $\mathcal{L}({}^{m}X)$ and $\mathcal{L}^{s}({}^{m}X)$ in place of $\mathcal{P}({}^{m}X; \mathbb{K})$, $\mathcal{L}({}^{m}X; \mathbb{K})$ and $\mathcal{L}^{s}({}^{m}X; \mathbb{K})$, respectively. More generally, a map
$P: X\rightarrow Y$ is a {\em continuous polynomial of degree $\leq m$} if
\[
P=P_0 +P_1 +\cdots +P_m\,,
\]
where $P_k \in\mathcal{P}({}^{k}X; Y)$, $1\leq k\leq m$, and $P_0: X\rightarrow Y$ is a constant function. The space of continuous polynomials from $X$ to $Y$ of degree at most $m$ is denoted by $\mathcal{P}_m(X; Y)$. If $Y=\mathbb{K}$, then we use the notation $\mathcal{P}_{m}(X)$ instead of $\mathcal{P}_{m}(X; \mathbb{K})$. We define the norm of a continuous (homogeneous) polynomial $P: X\rightarrow Y$ by
\[
\|P\|_{B_X}=\sup\{\|P(x)\|_Y:\ x\in B_X\}\,.
\]
Similarly, if $L: X^{m} \rightarrow Y$ is a continuous $m$-linear mapping we define its norm by
\[
\|L\|_{B_{X}^{m}}=\sup\{\|L(x_1, \ldots, x_m)\|_Y: \ x_1, \ldots, x_m\in B_X\}\,.
\]
When convenient we shall denote $\|L\|_{B_{X}^{m}}$ by $\|L\|$ and $\|P\|_{B_{X}}$ by $\|P\|$. Note that $\mathcal{P}({}^m X; Y)$ and $\mathcal{L}({}^m X; Y)$ are Banach spaces. Finally, observe that by the Hahn-Banach theorem we may restrict attention to the case where $Y=\mathbb{R}$ or $\mathbb{C}$, since estimates can then be tranferred to arbitrary Banach spaces as required.

If $P\in\mathcal{P}_m(X; Y)$ and $x\in X$, then $D^{k}P(x)$, $2\leq k\leq m$, denotes the $k$th Fr\'echet derivative of $P$ at $x$. Recall that $D^{k}P(x)$ would be, in fact, a symmetric $k$-linear mapping on $X^{k}$, whose associated $k$-homogeneous polynomial will be represented by $\widehat{D}^{k}P(x)$. So, $\widehat{D}^{k}P(x):= \widehat{D^{k}P(x)}$. We just write $DP(x)$ for the first Fr\'echet derivative of $P$ at $x$. If $\widehat{L}\in\mathcal{P}(^{m}X; Y)$, for any vectors $x, y_{1}, \ldots ,y_{k}$ in $X$ and any $k\leq m$ the following identity (see for instance \cite[7.7 theorem]{Chae}) holds
\begin{equation}\label{k-differential}
\frac{1}{k!}D^{k}\Lh(x)(y_{1}, \ldots, y_{k})={m\choose k}L(x^{m-k}, y_{1}, \ldots, y_{k})\,.
\end{equation}
In particular, for $x, y\in X$
\begin{equation}\label{k-homog-differential}
\frac{1}{k!}\Dh^{k}\Lh(x)y={m\choose k}L(x^{m-k} y^{k})
\end{equation}
and for $k=1$
\begin{equation}\label{1-differential}
D\Lh(x)y=\Dh\Lh(x)y=mL(x^{m-1} y)\,.
\end{equation}
Here, $L(x^{m-k} y^{k})$ denotes $L(\underbrace{x, \ldots, x}_{(m-k)}, \underbrace{y, \ldots, y}_{k})$. For general background on polynomials, we refer to \cite{Chae} and \cite{D}.

Let $r_{n}(t):=\text{sign}(\sin2^n\pi t)$ be the \textit{$n$th Rademacher function} on [0,1]. The Rademacher functions $(r_{n})$ form an orthonormal set in $L^{2}[0, 1], dt)$ where $dt$ denotes Lebesgue measure on [0,1]. If $L\in{\mathcal L}^{s}({}^{m} X)$, the next formula expresses a well known {\it polarization formula} in a very convenient form (see \cite[Lemma 2]{Sar2}):
\begin{equation}\label{polarformula}
L(x_1, \ldots, x_m) = \frac{1}{m!}\int_{0}^{1} r_{1}(t)\cdots r_{m}(t) \Lh\bigg(\sum_{n=1}^{m}r_{n}(t)x_{n}\bigg)\, dt\,.
\end{equation}
Therefore, each $\Lh\in\mathcal{P}({}^{m}X)$ is associated with a unique $L\in{\mathcal L}^{s}({}^{m} X)$ with the property that $\Lh(x)=L(x, \ldots, x)$. In many circumstances \cite{Davie-1,Davie-2,Peller,Var} it is of interest to compare the norm of $L\in\mathcal{L}^{s}(^{m}X)$ with the norm of  $\Lh\in\mathcal{P}({}^{m}X)$. For every $L\in{\mathcal L}^{s}({}^{m} X)$ it follows from \eqref{polarformula}(see \cite{D}) that
\begin{equation}\label{polarrange}
\|\Lh\|\leq \|L\|\leq \frac{m^{m}}{m!}\|\Lh\|\,.
\end{equation}
However, the right hand inequality can be tightened for many Banach spaces,  see for instance \cite{D,H1,Sar2}. For instance, the space of $1$-summable sequences $\ell^1$ and its finite-dimensional versions are particularly simple, yet fundamental, examples of Banach spaces where the constant $\frac{m^{m}}{m!}$ is sharp. We refer to \cite[corollary 1] {H1}, \cite[example 1]{Sar2}, \cite{Sar3}, \cite[example $1.39$]{D} and \cite{KST}.
We introduce the following definition.

\begin{definition}
Let $X$ be a Banach space over $\mathbb{K}$, where $\mathbb{K}=\mathbb{R}$, or $\mathbb{C}$.
If $k_1, \ldots, k_n$ are nonnegative integers whose sum is $m$, let
\begin{eqnarray*}
\lefteqn{\mathbb{K}(k_1, \ldots, k_n; X)}\\
& & =\inf \left\{M>0: \sup_{x_1, \ldots, x_n\in B_X}|L(x_{1}^{k_1}\ldots x_{n}^{k_n})|\leq M\|\Lh\|,\,\,\forall L\in\mathcal{L}^{s}({}^{m}X; \mathbb{K})\right\}\,.
\end{eqnarray*}
In the special case $k_1= \cdots =k_n=1$, we have $n=m$ and we let
\[
\mathbb{K}(m, X)=\inf \left\{M>0: \|L\|\leq M \|\Lh\|,\,\,\forall L\in\mathcal{L}^{s}({}^{m}X; \mathbb{K})\right\}\,.
\]
We call $\mathbb{K}(k_1, \ldots, k_n; X)$ and $\mathbb{K}(m, X)$ the \textit{$m$-th polarization constant} of the space $X$.
\end{definition}

Notice that $L(x_{1}^{k_1}\ldots x_{n}^{k_n})$ denotes $L(\underbrace{x_1, \ldots, x}_{k_1}, \ldots, \underbrace{x_{n}, \ldots, x_{n}}_{k_n})$. We shall write $\mathbb{R}(k_1, \ldots, k_n; X)$, $\mathbb{C}(k_1, \ldots, k_n; X)$ instead of $\mathbb{K}(k_1, \ldots, k_n; X)$, if the space $X$ is real, complex, respectively.

The \textit{polarization constant} of the space $X$, see \cite[definition 1.40]{D}, is defined by
\[
\mathbb{K}(X):=\limsup_{m\to \infty} \mathbb{K}(m, X)^{1/m}
\]
and describes how the the $m$-th polarization constant of $X$ behaves asymptotically. From \eqref{polarrange} and Stirling's formula, $1\leq\mathbb{K}(X)\leq e$ for any Banach space $X$.

In the sequel, $H$ will denote a Hilbert space. A famous result, investigated by Banach \cite{Banach} and many other authors, for example \cite{BS,CS,H1,Ho,Ke,PST}, asserts that $\mathbb{K}(m, H)=1$. In other words, $\|L\|=\|\Lh\|$ for every $L\in\mathcal{L}^{s}(^{m} H)$. In fact it was shown in \cite{BSar} that this is a characteristic property of \textit{real} Hilbert spaces. We also have $\mathbb{K}(k_1, \ldots, k_n; H)=1$, where $k_1, \ldots, k_n$ are nonnegative integers whose sum is $m$. Obviously $|L(x_{1}^{k_1}\ldots x_{n}^{k_n})|\leq \|\Lh\|$ for every $L\in\mathcal{L}^{s}(^{m} H)$, where $x_1, \ldots, x_n\in B_H$.

Since from \eqref{1-differential} $D\Lh(x)(y)=m L(x^{m-1} y)$, $y\in H$, to prove $\|L\|=\|\Lh\|$ by an inductive argument, it suffices to show that $|L(x^{m- 1} y)|\leq\|\Lh\|$ for any unit vectors $x$ and $y$ in $H$. In other words, $\|L\|=\|\Lh\|$ for any $\Lh\in\mathcal{P}\left({}^{m} H\right)$ if and only if
\begin{equation}\label{Bernstein-1}
\|D\Lh\|\leq m\|\Lh\|\,,\qquad \forall\, \Lh\in\mathcal{P}\left({}^{m} H\right)\,.
\end{equation}

Banach proved this result for continuous symmetric $m$-linear forms and continuous $m$-homogeneous polynomials on {\em finite dimensional real} Hilbert spaces. The proof works equally well for real and complex Hilbert spaces, and the condition of finite dimensionality is only needed to ensure that the $m$-linear form attains its norm.  The result that $\|L\|=\|\Lh\|$ is true for all Hilbert spaces, and, as pointed out by Banach, can be obtained through a simple limit argument based on the finite dimensional case.

Clearly, if $\Lh$ attains its norm at $x_{0}\in B_{H}$ then $L$ also attains its norm at $\left(x_{0}, \ldots, x_{0}\right)\in B_{H}^{m}$. When $H$ is finite dimensional, $L$ will always attain its norm, since the closed unit ball of $H$ is compact. However, when $H$ is infinite dimensional, $L$ need not attain its norm: if $H=\ell^{2}$, the space of square summable sequences, and $L(x, y) =\sum_{n=1}^{\infty} \frac{n}{n+1}x_{n}y_{n}$, it is easy to see that $\|L\|=1$, but that $|L(x, y)|< 1$ for all unit vectors $x=(x_{n})$ and $y=(y_{n})$ in $H$.

It is true, but not obvious, that if $L$ attains its norm at $\left(x_{1}, \ldots, x_{m}\right)\in B_{H}^{m}$, then $\Lh$ also attains its norm at some $x_{0}\in B_{H}$. When $L$ does attain its norm, an explicit construction has been given in \cite[section $2$]{PST} to provide a unit vector $x_{0}$ with $\|\Lh\|=|\Lh(x_{0})|$.


\section{\bf Preliminary results.}

It is simple to verify that for isometric Banach spaces $X, Y$ we have $\mathbb{K}(k_1, \ldots, k_n; X)
=\mathbb{K}(k_1, \ldots, k_n; Y)$. If $M$ is a closed subspace of the Banach space $X$, then it is routine to prove that $\mathbb{K}(k_1, \ldots, k_n; X/M)\leq\mathbb{K}(k_1, \ldots, k_n; X)$ and in particular $\mathbb{K}(m; X/M)\leq\mathbb{K}(m; X)$. However, the relationship between $\mathbb{K}(k_1, \ldots, k_n; X)$ and $\mathbb{K}(k_1, \ldots, k_n; Y)$ is less satisfactory. We state one result, but omit the proof which is quite similar to an analogous argument used in the proof of lemma $1.46$ in \cite{D}.

\begin{lemma}\label{projection}
If $M$ is a closed subspace of the Banach space $X$ and there exists a continuous projection $\pi$ of $X$ onto $M$, then
\[
\mathbb{K}(k_1, \ldots, k_n; M)\leq \|\pi\|^{m}\mathbb{K}(k_1, \ldots, k_n; X)\,.
\]

In particular if $M$ is a $1$-complemented subspace of $X$, then $\mathbb{K}(k_1, \ldots, k_n; M) \leq\mathbb{K}(k_1, \ldots, k_n; X)$.
\end{lemma}

If $k_1, \ldots, k_n$ are nonnegative integers whose sum is $m$, for $L^{p}(\mu)$ spaces we also set
\[
\mathbb{K}(k_{1}, \ldots, k_{n}; p)=\sup\{\mathbb{K}(k_{1}, \ldots, k_{n}; L^{p}(\mu)): \text{$\mu$ is a measure}\}\,.
\]
Using Lemma \ref{projection} we can easily see that if $1\leq p<\infty$
\begin{equation}\label{projappl-1}
\mathbb{K}(k_{1}, \ldots, k_{n}; p)=\mathbb{K}(k_{1}, \ldots, k_{n}; L^{p}(\mu))
\end{equation}
for any $\mu$ with $L^{p}(\mu)$ \textit{infinite-dimensional} (we refer to \cite{Sar2}).

It is known, see \cite[theorem $II.3.14$]{LT}) or \cite[proposition $11.1.9$]{AK}), that we can embed $L^{p}[0, 1]$ isometrically into $L^{r}[0, 1]$ for $1\leq r\leq p\leq 2$. If $r'$ and $p'$ are the conjugate exponents of $r$ and $p$, respectively, it follows from \eqref{projappl-1} and Lemma \ref{projection} that
\[
\mathbb{K}(k_1, \ldots, k_n; p')\leq\mathbb{K}(k_1, \ldots, k_n; r')\,,\qquad 2\leq p'\leq r'\leq \infty\,.
\]
Thus $\mathbb{K}(k_1, \ldots, k_n; p)$ is an increasing function of $p$ over the range $2\leq p\leq \infty$.

Let $X$ and $Y$ be two isomorphic Banach spaces. The Banach-Mazur distance between $X$ and $Y$ denoted as $d(X, Y)$ is defined as
\[
d(X, Y)=\inf\{\|T\|\|T^{-1}\|:\,T: X\xrightarrow{\text{onto}} Y\}\,.
\]
If $X$ is finite dimensional and $d(X, Y)=1$, then $X$ is isometric to $Y$.

Using the Banach-Mazur distance between $X$ and $Y$, we relate $\mathbb{K}(k_1, \ldots, k_n; X)$ and $\mathbb{K}(k_1, \ldots, k_n; Y)$.

\begin{lemma}\label{BM-distance}
If $X$, $Y$ are isomorphic Banach spaces, then
\[
\mathbb{K}(k_1, \ldots, k_n; Y)\leq (d(X, Y))^{m}\,\mathbb{K}(k_1, \ldots, k_n; X)\,.
\]
In particular,
\[
\mathbb{K}(m, Y)\leq (d(X, Y))^{m}\,\mathbb{K}(m, X)\,.
\]
\end{lemma}

The proof of Lemma \ref{BM-distance} is quite similar to the proof of lemma $12$ in \cite{BST}, we spare the reader the details.

In the case of $L^p(\mu)$ spaces, Lemma \ref{BM-distance} allows us to give an estimate for $\mathbb{K}(k_1, \ldots, k_n; L^p(\mu))$.

\begin{proposition}\label{1st-estimate}
Let $k_1, \ldots, k_n$ be nonnegative integers whose sum is $m$. Then, for the $L^p(\mu)$ space,
$1\leq p\leq \infty$, we have
\begin{equation}\label{1st-estimate1}
\mathbb{K}(k_1, \ldots, k_n; L^p(\mu))\leq n^{m|\frac{1}{2}-\frac{1}{p}|}\,.
\end{equation}
If $X$ is any Banach space
\begin{equation}\label{1st-estimate2}
\mathbb{K}(k_1, \ldots, k_n; X)\leq n^{\frac{m}{2}}\,.
\end{equation}
In particular, if $k_1= \cdots =k_n=1$, then $n=m$ and we have
\begin{equation}\label{1st-estimate3}
\mathbb{K}(m, L^p(\mu))\leq m^{m|\frac{1}{2}-\frac{1}{p}|}
\end{equation}
and
\begin{equation}\label{1st-estimate4}
\mathbb{K}(m; X)\leq m^{\frac{m}{2}}\,.
\end{equation}
\end{proposition}

\begin{proof}
Let $\varepsilon>0$. Choose $L\in\mathcal{L}^{s}(^{m} L^{p}(\mu); \mathbb{K})$ and $f_{i}\in L^{p}(\mu)$, $\|f_{i}\|_{p}\leq 1$, $i=1, \ldots, n$, such that
\[
|L(f_{1}^{k_1} \ldots f_{n}^{k_n})|\geq(\mathbb{K}(k_1, \ldots, k_n; L^p(\mu))-\varepsilon)\|\Lh\|\,.
\qquad\qquad\qquad\qquad(\ast)
\]
If $N:=\mathrm{span}\{f_{1},\ldots, f_{n}\}$, then $N$ is an $n$-dimensional subspace of $L^p(\mu)$ and from \cite{Lewis}, see also corollary III.B.9 in Wojtaszczyk's book \cite{Woj}, $d(N, \ell_n^{2})\leq n^{|\frac{1}{2}-\frac{1}{p}|}$, where $\ell_n^{2}$ is the $n$-dimensional Hilbert space. If $F=L\left|_{N^m}\right.$, then $\Fh$ is the restriction of $\Lh$ to $N$. Since $\|\Fh\|\leq\|\Lh\|$, it follows from ($\ast$) that
\[
\mathbb{K}(k_1, \ldots, k_n; N)\|\Fh\|\geq|F(f_{1}^{k_1} \ldots f_{n}^{k_n})|\geq\mathbb{K}(k_1, \ldots, k_n; L^p(\mu))-\varepsilon)\|\Fh\|
\]
and therefore
\[
\mathbb{K}(k_1, \ldots, k_n; L^p(\mu))-\varepsilon\leq\mathbb{K}(k_1, \ldots, k_n; N)\,. \qquad\qquad\qquad\qquad(\ast\ast)
\]
But $\mathbb{K}(k_1, \ldots, k_n; \ell_n^{2})=1$ and from Lemma \ref{BM-distance}
\[
\mathbb{K}(k_1, \ldots, k_n; N)\leq d(N, \ell_n^{2})^{m}\,\mathbb{K}(k_1, \ldots, k_n; \ell_n^{2}))
\leq n^{m|\frac{1}{2}-\frac{1}{p}|}\,.
\]
Hence, inequality ($\ast\ast$) implies
\[
\mathbb{K}(k_1, \ldots, k_n; L^p(\mu))-\varepsilon\leq n^{m|\frac{1}{2}-\frac{1}{p}|}
\]
and this proves \eqref{1st-estimate1}.

Since for any $n$-dimensional subspace $N$ of a Banach space $X$ we have $d(N, \ell_n^{2})\leq \sqrt{n}$ (see corollary III.B.9 in \cite{Woj}), the proof of \eqref{1st-estimate2} is quite similar.
\end{proof}

The estimates for the polarization constants in Proposition \ref{1st-estimate} are not in principle the optimal ones. Indeed, the upper bounds that we get as in \eqref{1st-estimate3} and \eqref{1st-estimate4} are not as sharp as the estimate $\frac{m^{m}}{m!}$; see inequality \eqref{polarrange} in the Introduction. However, since $\mathbb{K}(m; H)=1$ for any Hilbert space $H$, we have  $\mathbb{K}(m; L^2(\mu))=1$ and for $p$ close to $2$ the constant in \eqref{1st-estimate3} is close to optimal. Therefore, for $p$ close to $2$ we improve the estimate given in proposition $3.7$ in \cite{ChSar}.

\begin{proposition}$($\cite[proposition 3.7]{ChSar}$)$\label{gpolar-constants}
For the $m$th polarization constant $\mathbb{K}(m, p)$, $m\geq 2$, we have the estimates
\begin{equation}\label{gpolar-constants-1}
\mathbb{K}(m; p)\leq \begin{cases} \frac{m^{m/p}}{m!} &\text{$1\leq p\leq m'$ ,}\\
\min\{\frac{m^{m/m'}}{m!},\, m^{m|p-2|/2p}\} &\text{$m'\leq p\leq m$ ,}\\
\frac{m^{m/p'}}{m!} &\text{$m\leq p\leq\infty$ .}
                                \end{cases}
\end{equation}
In the case $1\leq p\leq m'$ the estimate is best possible.
\end{proposition}

Here, as usual, $m'=m/(m-1)$ and $p'=p/(p-1)$ are the conjugate exponents of $m$ and $p$, respectively.
Harris\cite{H1} showed that for a complex $L^p(\mu)$ space with $1\leq p\leq\infty$
\begin{equation}\label{H-1}
\mathbb{C}(m; p)\leq\left(\frac{m^{m}}{m!}\right)^{\frac{|p-2|}{p}}
\end{equation}
provided that $m$ is a power of $2$. Harris conjectured that \eqref{H-1} holds for all positive integers $m$. For $1\leq p\leq m'$ and $m\leq p\leq\infty$ and for any real or complex $L^p(\mu)$ space, the constants in \eqref{gpolar-constants-1} improve the constant conjectured by Harris for all positive integers $m$.


\section{\bf Estimates for polynomial norms on \textit{complex} Banach spaces.}

We shall need  the Bochner integral, see \cite{DU}. The basis for this is a measure space $(\Omega, \Sigma, \mu)$ and a Banach space X. A function $s:\Omega\rightarrow X$ is called {\it simple} if there exist $x_1, \ldots, x_n \in X$ and  $E_1, \ldots, E_n \in \Sigma$ such that $s=\sum_{i=1}^{n} x_{i}\chi_{E_i}$, where $\chi_{E_i}(\omega)=1$ if $\omega\in E_i$ and $\chi_{E_i}(\omega)=0$ if $\omega\notin E_i$. A function $f: \Omega\rightarrow X$ is called {\it $\mu$-measurable} if there exist a sequence of simple functions $(s_{n})$ with $\lim_{n\rightarrow \infty} \|s_{n}-f\|=0$. A  $\mu$-measurable function $f: \Omega\rightarrow X$ is {\it $\mu$-Bochner integrable} if
$\int_{\Omega} \|f\|\, d\mu<\infty$.

If $1\leq p<\infty$, the symbol $L_{\mu}^{p}(\Omega, \Sigma, \mu, X) (=L_{\mu}^{p}(X))$ will stand for all(equivalence classes of) $\mu$-Bochner integrable functions $f: \Omega\rightarrow X$ such that
\[
\|f\|_{p}=\left(\int_{\Omega} \|f\|^{p}\, d\mu\right)^{1/p}<\infty\,.
\]
Normed by the functional $\|\cdot\|_p$ defined above $L_{\mu}^{p}(X)$ becomes a Banach space. The symbol $L_{\mu}^{\infty}(\Omega, \Sigma, \mu, X) (=L_{\mu}^{\infty}(X))$ will stand for all(equivalence classes of) essentially bounded $\mu$-Bochner integrable functions $f: \Omega \to X$. Normed by the functional $\|\cdot\|_{\infty}$ defined for $f\in L_{\mu}^{\infty}(X)$ by
\[
\|f\|_{\infty}=\text{ess} \sup\{\|f(\omega)\|:\, \omega\in\Omega\}<\infty\,,
\]
$L_{\mu}^{\infty}(X)$ becomes a Banach space. The completion in $L_{\mu}^{\infty}(X)$ of the simple functions $s=\sum_{i=1}^{n} x_{i}\chi_{E_i}$ with $\mu(E_i)<\infty$ for every $i=1, \ldots, n$, is denoted by $L_{\mu}^{\infty, 0}(X)$.

When $X=\mathbb{K}$, $\mathbb{K}=\mathbb{R}$ or $\mathbb{C}$, we use the symbol $L^{p}(\mu)$ for $L_{\mu}^{p}(X)$, $1\leq p\leq\infty$. If $(\Omega, \Sigma, \mu)$ is the usual Lebesgue measure on
$[0, 1]$ we denote $L^{p}(\mu)$ by $L^{p}[0, 1]$.

Now we consider one special case: Given any set $\Gamma$, we define $\ell^{p}(\Gamma)=L^{p}(\Gamma, 2^{\Gamma}, \mu)$, where $\mu$ is counting measure on $\Gamma$. What this means is that we identify functions $f: \Gamma\rightarrow \mathbb{K}$ with ``sequences" $x=(x_{\gamma})$ in the usual way:
$x_{\gamma}=f(\gamma)$, and we define
\[
\|x\|_{p}=\bigg(\sum_{\gamma \in\Gamma} |x_{\gamma}|^{p}\bigg)^{1/p}=\left(\int_{\Gamma} |f(\gamma)|^{p}\, d\mu(\gamma)\right)^{1/p}=\|f\|_{p}\,.
\]
Please note that if $x\in \ell^{p}(\Gamma)$, then $x_{\gamma}=0$ for all but countably many $\gamma$. For $p=\infty$, we set
\[
\|x\|_{\infty}=\sup_{\gamma \in\Gamma} |x_{\gamma}|=\sup_{\gamma \in\Gamma} |f(\gamma)|=\|f\|_{\infty}\,.
\]
If $\Gamma=\mathbb{N}$, we denote $\ell^{p}(\Gamma)$ by $\ell^{p}$. We write $\ell_{n}^{p}$ to denote
$\mathbb{K}^{n}$ under the $\ell^{p}$ norm.
\\
\\
We now mention a formula that will be useful to us.

\textit{Binomial formula}: Let $X$ be a Banach space and let $x_1, \ldots, x_n\in X$. Then for any $L\in\mathcal{L}^{s}({}^{m}X; \mathbb{K})$
\begin{equation}\label{Binomial}
\Lh(x_1+ \cdots +x_n)=\sum_{j_1+ \cdots +j_n=m} \frac{m!}{j_{1}! \cdots j_{n}!} L(x_{1}^{j_1}\ldots x_{n}^{j_n})\,.
\end{equation}

Recall the classical Clarkson inequalities: If $f_{1}\ldots f_{n}\in L^{p}(\mu)$, then
\begin{eqnarray}\label{GenClarkson}
\begin{split}
&\bigg(\int_{0}^{1} \big\|\sum_{j=1}^{n} f_{j}r_{j}(t)\big\|_{p}^{p'}\, dt\bigg)^{1/p'} &\leq& \bigg(\sum_{j=1}^{n} \|f_{j}\|_{p}^{p}\bigg)^{1/p}\qquad &\text{for}\quad 1\leq p\leq 2\,,\\
&\text{and}\\
&\bigg(\int_{0}^{1} \big\|\sum_{j=1}^{n} f_{j}r_{j}(t)\big\|_{p}^{p}\, dt\bigg)^{1/p} &\leq& \bigg(\sum_{j=1}^{n} \|f_{j}\|_{p}^{p'}\bigg)^{1/p'}\qquad &\text{for}\quad 2\leq p<\infty\,.
\end{split}
\end{eqnarray}
We refer to \cite{WW} for this and other similar $L^{p}$-inequalities.

Now we prove a \textit{generalized Clarkson inequality} which is a standard type of interpolation lemma.

\begin{lemma}\label{GenClarkson}
Let $f_{1}\ldots f_{n}\in L^{p}(\mu)$ and let $\la$ be Haar measure on $\mathbb{T}^n$, the $n$-fold product of the circle group $\mathbb{T}=\{z\in\mathbb{C}:\, |z|=1\}$. Thus $d\la(\theta)=(1/2\pi)^{n}d\theta_{1}\cdots d\theta_{n}$. Then
\begin{eqnarray}\label{GenClarkson-1}
\begin{split}
&\bigg(\,\int\limits_{\mathbb{T}^n} \big\|\sum_{j=1}^{n} f_{j}e^{i\theta_j}\big\|_{p}^{p'}\, d\la(\theta)\bigg)^{1/p'} &\leq& \bigg(\sum_{j=1}^{n} \|f_{j}\|_{p}^{p}\bigg)^{1/p} \qquad &\text{for}\quad 1\leq p\leq 2\,,\\
&\text{and}\\
&\bigg(\,\int\limits_{\mathbb{T}^n} \big\|\sum_{j=1}^{n} f_{j}e^{i\theta_j}\big\|_{p}^{p}\, d\la(\theta)\bigg)^{1/p} &\leq& \bigg(\sum_{j=1}^{n} \|f_{j}\|_{p}^{p'}\bigg)^{1/p'}\qquad &\text{for}\quad 2\leq p<\infty\,.
\end{split}
\end{eqnarray}
\end{lemma}

\begin{proof}
Consider the linear operator
\[
T: \ell_{n}^{2}(L^{2}(\mu))\longrightarrow L_{\la}^{2}(L^{2}(\mu))
\]
defined by
\[
T: f=(f_{1}, \ldots f_{n})\longmapsto f_{1}e^{i\theta_1}+\cdots +f_{n}e^{i\theta_n}\qquad\qquad\qquad \qquad (\ast)
\]
where $f_{j}\in L^{2}(\mu)$, $j=1, \ldots, n$. We have
\[
\|Tf\|=\bigg(\,\int\limits_{\mathbb{T}^n} \big\|\sum_{j=1}^{n} f_{j}e^{i\theta_j}\big\|_{2}^{2}\, d\la(\theta)\bigg)^{1/2}= \bigg(\sum_{j=1}^{n} \|f_{j}\|_{2}^{2}\bigg)^{1/2}\,.
\]
Now if the linear operator
\[
T: \ell_{n}^{1}(L^{1}(\mu))\longrightarrow L_{\la}^{\infty, 0}(L^{1}(\mu))
\]
is defined as in ($\ast$), where $f_{j}\in L^{1}(\mu)$, $j=1, \ldots, n$, we have
\[
\|Tf\|=\sup_{\theta_{1}, \ldots,\theta_{n}} \bigg\|\sum_{j=1}^{n} f_{j}e^{i\theta_j}\bigg\|_{1}\leq
\sum_{j=1}^{n} \|f_{j}\|_{1}\,.
\]
Hence by an extended version of the Riesz-Thorin interpolation theorem \cite[theorems 4.1.2, 5.1.1, 5.1.2]{BL} the first inequality in \eqref{GenClarkson-1} holds for $1\leq p\leq 2$.

The proof of the second inequality in \eqref{GenClarkson-1} is similar. Simply, instead of the linear operator $T: \ell_{n}^{1}(L^{1}(\mu))\longrightarrow L_{\la}^{\infty, 0}(L^{1}(\mu))$ consider the linear operator
\[
T: \ell_{n}^{1}(L^{\infty}(\mu))\longrightarrow L_{\la}^{\infty, 0}(L^{\infty}(\mu))\,,
\]
where $f_{j}\in L^{\infty}(\mu)$, $j=1, \ldots, n$, with
\[
\|Tf\|=\sup_{\theta_{1}, \ldots,\theta_{n}} \bigg\|\sum_{j=1}^{n} f_{j}e^{i\theta_j}\bigg\|_{\infty} \leq\sum_{j=1}^{n} \|f_{j}\|_{\infty}\,.
\]

\end{proof}

The first inequality in \eqref{GenClarkson-1} is lemma I.2.4 in \cite{Sar1} or lemma 1 in \cite{Sar2}. We use inequality\eqref{GenClarkson-1} in the proof of the following theorem.

\begin{theorem}\label{Mthm}
Let $k_1, \ldots, k_n$ be nonnegative integers whose sum is $m$. Then for any complex $L^{p}(\mu)$ space
\begin{equation}\label{Mthm-1}
\mathbb{C}(k_1, \ldots, k_n; L^{p}(\mu))\leq \begin{cases} \left(\frac{m^{m}}{k_{1}^{k_1}\cdots k_{n}^{k_n}}\right)^{1/p}\,\frac{k_{1}!\cdots k_{n}!}{m!} &\text{$1\leq p\leq m'$ ,}\\
\left(\frac{m^{m}}{k_{1}^{k_1}\cdots k_{n}^{k_n}}\right)^{1/p'}\,\frac{k_{1}!\cdots k_{n}!}{m!} &\text{$m\leq p\leq\infty$ .}
                                \end{cases}
\end{equation}
In the case $1\leq p\leq m'$ the estimate is best possible.
\end{theorem}

\begin{proof}
Let $L\in\mathcal{L}^{s}(^{m} L^{p}(\mu); \mathbb{C})$. If $x_{1},\ldots, x_{n}$ are unit vectors in $L^{p}(\mu)$, put $r_j=(k_j/m)^{1/p}$ for $j=1, \ldots, n$ and define
\[
f(z_{1},\ldots, z_{n}):= \Lh(r_{1}x_{1}z_{1}+\cdots +r_{n}x_{n}z_{n})
\]
for $z_{1},\ldots, z_{n}\in\mathbb{C}$. From the multinomial formula we have
\[
f(z_{1},\ldots, z_{n})=\sum_{j_1+ \cdots +j_n=m} \frac{m!}{j_{1}! \cdots j_{n}!}r_{1}^{j_1}\cdots r_{n}^{j_n}L(x_{1}^{j_1}\ldots x_{n}^{j_n})z_{1}^{j_1}\cdots z_{n}^{j_n}\,.
\]
Subsequently partial differentiation yields
\[
\frac{\partial^{k_1 +\cdots+ k_n}}{\partial{z_{1}^{k_1}}\ldots \partial{z_{n}^{k_n}}}f(0,\ldots, 0)
=m!r_{1}^{k_1}\cdots r_{n}^{k_n}L(x_{1}^{k_1}\ldots x_{n}^{k_n})\,.
\]
Since $f$ is a homogeneous polynomial on $\mathbb{C}^{n}$, by Cauchy's differentiation formula
\[
\frac{\partial^{k_1 +\cdots+ k_n}}{\partial{z_{1}^{k_1}}\ldots \partial{z_{n}^{k_n}}}f(0,\ldots, 0)
=\frac{k_{1}!\cdots k_{n}!}{(2\pi i)^{n}}\idotsint\limits_{\mathbb{T}^{n}} \frac{f(z_{1},\ldots, z_{n})}{z_{1}^{k_{1}+1}\ldots z_{n}^{k_{n}+1}}\,dz_{1}\cdots dz_{n}
\]
where $\mathbb{T}^{n}$ is the $n$-fold product of the circle group. So
\[
|L(x_{1}^{k_1}\ldots x_{n}^{k_n})|\leq\frac{1}{r_{1}^{k_1}\cdots r_{n}^{k_n}}\frac{k_{1}!\cdots k_{n}!}{m!}\|\Lh\|\int\limits_{\mathbb{T}^{n}} \big\|\sum_{j=1}^{n} r_{j}x_{j}e^{i\theta_{j}}\big\|_{p}^m\, d\la(\theta)
\]
where $d\la(\theta)=(1/2\pi)^{n}d\theta_{1}\cdots d\theta_{n}$. Now, using H\"{o}lder's inequality
and the first inequality in \eqref{GenClarkson-1}, we obtain the first inequality in \eqref{Mthm-1}.

For the proof of the second inequality in \eqref{Mthm-1}, put $r_j=(k_j/m)^{1/p'}$ for $j=1, \ldots, n$ and define as before $f(z_{1},\ldots, z_{n}):= \Lh(r_{1}x_{1}z_{1}+\cdots +r_{n}x_{n}z_{n})$, for $z_{1},\ldots, z_{n}\in\mathbb{C}$. Now repeat the previous proof using the second inequality in \eqref{GenClarkson-1}.

Finally, Example \ref{Mexample} shows that the constant in the first inequality in \eqref{Mthm-1} is best possible.
\end{proof}

Observe that in the complex case, the first and the third estimate in \eqref{gpolar-constants-1} is just inequality \eqref{Mthm-1} for $k_1= \cdots =k_n=1$. For some other interesting results related to $\mathbb{C}(k_1, \ldots, k_n; L^{p}(\mu))$ we refer to section 4 in \cite{DGR}.

\begin{remark}
The first inequality in \eqref{Mthm-1} is proposition I.2.6 in \cite{Sar1} or theorem $1$ in \cite{Sar2}. In fact, the first inequality in \eqref{Mthm-1} is an improvement of a special case of a result of Harris, see theorem $1$ in \cite{H1}. He showed that if $1\leq p\leq\infty$ and if $X$ is a complex normed space, then the first inequality in \eqref{Mthm-1} holds provided that
\[
\bigg\|\sum_{j=1}^{n} z_{j}x_{j}\bigg\|\leq \bigg(\sum_{j=1}^{n} |z_{j}|^{p}\bigg)^{1/p}\qquad\text{for all}\,\, (z_1, \ldots, z_n)\in \mathbb{C}^{n}\,.
\]
\end{remark}

A classical result of Banach and Mazur \cite{BM} states that every separable Banach space is isometric to a quotient of $\ell^{1}$. Similarly one can prove that given a Banach space $X$, there is a set $\Gamma$ such that $X$ is isometric to a quotient of $\ell^{1}(\Gamma)$. Hence, from Theorem \ref{Mthm} in the special case $p=1$ and from Lemma \ref{projection} we have

\begin{corollary}\label{Mthmcor1}
Let $k_1, \ldots, k_n$ be natural numbers whose sum is $m$. Then for any complex Banach space $X$
\begin{equation}\label{Mthmcor-1}
\mathbb{C}(k_1, \ldots, k_n; X)\leq\frac{m^m}{k_{1}^{k_1}\cdots k_{n}^{k_n}}\,\frac{k_{1}!\cdots k_{n}!}{m!}
\end{equation}
and the constant is best possible.
\end{corollary}
This last result also follows from theorem $1$ in \cite{H1}. By the use of the ``generalized Rademacher functions", another proof of the same result can be found in \cite{ALRT}. We also refer to corollary $4$ in \cite{H2} which is an application of corollary $3$ in \cite{H2} or proposition $3.4$ in \cite{ChSar}. In Example \ref{real-2-2} we show that \eqref{Mthmcor-1} does not hold when {\it real} Banach spaces are considered.

The following Example \ref{Mexample} (see also example 1 in \cite{Sar2}) shows that for $1\leq p\leq m'$ the constant in \eqref{Mthm-1} is best possible. Hence, the constant in \eqref{Mthmcor-1} is best possible.

\begin{example}\label{Mexample}
Consider the symmetric $m$-linear form $L$ on the real or complex sequence space $\ell^{p}$ defined by
\[
L(x_1, \ldots, x_m)=\frac{1}{m!}\sum_{\sigma\in S_m} x_{1\sigma(1)}\cdots x_{m\sigma(m)}\,,
\]
where $x_i=(x_{in})_{n=1}^{\infty}$, $i=1, \ldots, m$, and $S_m$ is the set of permutations of the first $m$ natural numbers. Then, $\Lh(u)=u_{1}\cdots u_{m}$, $u=(u_{i})$, is the $m$-homogeneous polynomial associated to $L$ end
\[
L(x_{1}^{k_1}\ldots x_{n}^{k_n})=\frac{1}{m!}\sum_{\sigma\in S_m} x_{1\sigma(1)}\cdots x_{1\sigma(k_1)}\cdots x_{n\sigma(k_{1}+\cdots +k_{n-1}+1)}\cdots x_{n\sigma(k_{1}+\cdots +k_{n})}\,,
\]
where $k_{1}+ \cdots +k_{n}=m$. If $(e_i)$ is the standard unit vector basis of $\ell^{p}$, define
\begin{align*}
y_{1} &= k_{1}^{-1/p}(e_1 +\cdots+ e_{k_1})\\
y_{2} &= k_{2}^{-1/p}(e_{k_{1}+1} +\cdots+ e_{k_{1}+k_{2}})\\
\vdots\\
y_{n} &= k_{n}^{-1/p}(e_{k_{1}+\cdots +k_{n-1}+1} +\cdots+ e_{k_{1}+\cdots +k_{n}})\,.
\end{align*}
It is easy to see that $y_{1},\ldots, y_{n}$ are unit vectors in $\ell^{p}$ and
\[
L(y_{1}^{k_1}\ldots y_{n}^{k_n})=\frac{1}{k_{1}^{k_{1}/p}\cdots k_{n}^{k_{n}/p}}\,\frac{k_{1}!\cdots k_{n}!}{m!}\,.
\]
On the other hand
\[
|\Lh(u)|=\{|u_{1}|^{p}\cdots|u_{m}|^{p}\}^{1/p}\leq\left\{\frac{|u_{1}|^{p}+\cdots +|u_{m}|^{p}}{m}\right\}^{m/p}
\]
by the arithmetic-geometric mean inequality and so $\|\Lh\|\leq 1/m^{m/p}$. Thus
\[
|L(y_{1}^{k_1}\ldots y_{n}^{k_n})|\geq\left(\frac{m^{m}}{k_{1}^{k_1}\cdots k_{n}^{k_n}}\right)^{1/p} \,\frac{k_{1}!\cdots k_{n}!}{m!}\|\Lh\|\,.
\]
Hence, for  $1\leq p\leq m'$
\[
|L(y_{1}^{k_1}\ldots y_{n}^{k_n})|=\left(\frac{m^{m}}{k_{1}^{k_1}\cdots k_{n}^{k_n}}\right)^{1/p} \,\frac{k_{1}!\cdots k_{n}!}{m!}\|\Lh\|\,.
\]
\end{example}

\vspace{0.2 cm}

Observe that in Example \ref{Mexample} the symmetric $m$-linear form $L$ can be defined on the real or complex $N$-dimensional space $\ell_{N}^{p}$, $m\leq N$. In particular, the latter means that for $X$ being a complex Banach space, the estimate
\[
\mathbb{C}(k_1, \ldots, k_n; X)\leq\mathbb{C}(k_1, \ldots, k_n; \ell^1)=\frac{m^m}{k_{1}^{k_1}\cdots k_{n}^{k_n}}\,\frac{k_{1}!\cdots k_{n}!}{m!}
\]
holds true.

Next, for any complex Banach space $X$ and for $n$ fixed we find the asymptotic growth of $\mathbb{C}(k_1, \ldots, k_n; X)$. For this we need Stirling's formula $m!\sim \sqrt{2\pi m}(m/e)^{m}$, where the sign $\sim$ means that the two quantities are asymptotic, that is their ratio tends to $1$ as $m$ tends to infinity. In particular, we shall make use of the following approximation for $k_{j}!$
\begin{equation}\label{Stirlingform}
\sqrt{2\pi}k_{j}^{k_{j}+1/2}e^{-k_{j}}\leq k_{j}!\leq ek_{j}^{k_{j}+1/2}e^{-k_{j}}\,,\qquad j=1, \ldots, n\,.
\end{equation}

\begin{proposition}\label{Cpolarconst}
Let $k_1, \ldots, k_n$ be natural numbers whose sum is $m$. For any complex Banach space $X$ and for $n$ fixed
\[
\limsup_{m\to \infty} \mathbb{C}(k_1, \ldots, k_n; X)^{1/m}=1\,.
\]
\end{proposition}

\begin{proof}
Since $\sqrt{2\pi}m^{m+1/2}e^{-m}\leq m!\leq em^{m+ 1/2}e^{-m}$, using the right-hand inequalities in \eqref{Stirlingform}, combined with the arithmetic-geometric mean inequality, we have
\begin{eqnarray*}
\frac{m^m}{k_{1}^{k_1}\cdots k_{n}^{k_n}}\,\frac{k_{1}!\cdots k_{n}!}{m!}
&\leq& \frac{e^n}{\sqrt{2\pi}m^{1/2}}(k_{1}\cdots k_{n})^{1/2}\\
&\leq& \frac{e^n}{\sqrt{2\pi}m^{1/2}}\left(\frac{k_{1}+\cdots +k_{n}}{n}\right)^{n/2}\\
&=& \frac{e^n}{\sqrt{2\pi}n^{n/2}}m^{(n-1)/2}\,.
\end{eqnarray*}
Taking into consideration \eqref{Mthmcor-1}, we conclude that
\[
\limsup_{m\to \infty} \mathbb{C}(k_1, \ldots, k_n; X)^{1/m}=1\,.
\]
\end{proof}

\begin{remark}
Dimant, Galicer and Rodr\'iguez have shown in \cite[theorem 1.1]{DGR} that for any \textit{finite dimensional complex} Banach space $X$, the polarization constant $\mathbb{C}(X)=1$. The key in the proof of their main theorem $1.1$, is proposition $2.1$ in \cite{DGR} which relies on \cite[proposition 4]{Sar4}. In proposition $2.1$ they show that the  polarization constant $\mathbb{C}(\ell_{d}^{1}) =\limsup_{m\to \infty} \mathbb{K}(m, \ell_{d}^{1})^{1/m} =1$. The alternative, and much shorter proof, of this result was communicated to them by one of the referees of \cite{DGR}. Our proof of Proposition \ref{Cpolarconst} is this alternative proof.

Contrary to the complex case, an example in \cite{DGR} shows that a \textit{finite dimensional real} Banach space can have polarization constant bigger that $1$.
\end{remark}

\vspace{0.2 cm}

{\bf An application of Theorem \ref{Mthm}:} Recall the following Markov-type inequality from \cite{ChSar}.

\begin{proposition}$($\cite[proposition 3.4]{ChSar}$)$\label{MthmMarkovineq}
Let $\Lh: L^{p}(\mu)\rightarrow \mathbb{C}$ be a continuous $m$-homogeneous polynomial, $m\geq 2$, on the {\it complex} $L^{p}(\mu)$ space. If $m'$ and $p'$ are the conjugate exponents of $m$ and $p$, respectively, for $k\leq m$ we have the following Markov-type inequality
\[
\|\Dh^{k}\Lh\|\leq C_{k, m}\|\Lh\|\,,
\]
where
\begin{equation}\label{MthmMarkovineq-1}
C_{k, m}= \begin{cases} \left(\frac{m^{m}}{(m-k)^{(m-k)}k^{k}}\right)^{1/p}k! &\text{$1\leq p\leq m'$ ,}\\
\left(\frac{m^{m}}{(m-k)^{(m-k)}k^{k}}\right)^{1/m'}k! &\text{$m'\leq p\leq m$ ,}\\
\left(\frac{m^{m}}{(m-k)^{(m-k)}k^{k}}\right)^{1/p'}k! &\text{$m\leq p\leq\infty$ .}
          \end{cases}
\end{equation}
In the case $1\leq p\leq m'$ the estimate is best possible.
\end{proposition}

If $x_{1}, x_{2}$ are unit vectors in $L^{p}(\mu)$, from identity \eqref{k-homog-differential} we have
\[
\Dh^{k}\Lh(x_{1})x_{2}=\frac{m!}{(m-k)!}L(x_{1}^{m-k}x_{2}^{k})\,.
\]
So, the first and the third estimate in \eqref{MthmMarkovineq-1} follow directly from Theorem \ref{Mthm}.

In general, if $X$ is a \textit{complex} Banach space and if $\Lh: X\rightarrow \mathbb{C}$ is a continuous $m$-homogeneous polynomial, for $k\leq m$ we have the following Markov-type inequalities:
\begin{equation}\label{complex-Markov}
\|\Dh^{k}\Lh\|\leq\frac{m^{m}}{(m-k)^{(m-k)}k^{k}}k!\|\Lh\|\,,\quad\|D^{k}\Lh\|
\leq\frac{m^{m}}{(m-k)^{(m-k)}}\|\Lh\|
\end{equation}
and the constants are best possible. This is an application of Corollary \ref{Mthmcor1} and Example \ref{Mexample} (we also refer to corollary 1 in \cite{H1}). Recall from inequality \eqref{polarrange} that $\|D^{k}\Lh\|\leq (k^k/k!)\|\Dh^{k}\Lh\|$.


\section{\bf Estimates for polynomial norms on \textit{real} Banach spaces.}

\subsection{\bf Using complexificarion of real Banach spaces.}

Let $k_1, \ldots, k_n$ be natural numbers whose sum is $m$ and let $X$ be a \textit{real} Banach space. If we complexify the real space $X$, using the estimate in \eqref{Mthmcor-1} we get an estimate for $\mathbb{R}(k_1, \ldots, k_n; X)$.

A \textit{complex} vector space $\widetilde{X}$ is a \textit{complexification} of
a \textit{real} vector space $X$ if the following two conditions hold:
\begin{itemize}
\item[(i)] there is a one-to-one real-linear map $j: X\rightarrow \widetilde{X}$ and

\item[(ii)] complex-span$\big(j(X)\big)=\widetilde{X}$.
\end{itemize}
If $X$ is a real vector space, we can make $X\times X$ into a complex vector space by
defining
\begin{align*}
(x,y)+(u,v) &:= (x+u, y+v)\quad\forall x,y,u,v\in X\,, \\
(\alpha+i\beta)(x, y) &:= (\alpha x-\beta y, \beta x+\alpha y)\quad\forall x,y\in X,\quad\forall \alpha, \beta \in\mathbb{R}.
\end{align*}
The map $j: X\rightarrow X\times X$; $x\mapsto (x, 0)$ clearly satisfies conditions (i) and (ii) above, and so this complex vector space is a complexification of $X$. It is convenient to denote it by
\[
\widetilde{X}=X\oplus iX\,.
\]
The norm on $\widetilde{X}$ can be specified by
\[
\|(x, y)\|=\|(|x|^2+|y|^2)^{1/2}\|\,, \qquad \forall x,y\in X.
\]
For more details, consult \cite[p. $326$]{DJT}.

If $X$ is a real-valued $L^{p}(\mu)$ space or $C(K)$ space, this complexification procedure yields the corresponding complex-valued space.

Bochnak and Siciak (see \cite[Theorem 3]{BS}) observed that when $X$ is a real Banach space, each $L\in\mathcal{L}(^{m}X; \mathbb{R})$ has a unique complex extension $\widetilde{L}\in\mathcal{L} (^{m}\widetilde{X}; \mathbb{C})$, defined by the formula
\[
\widetilde{L}(x_{1}^{0}+ ix_{1}^{1},\ldots ,x_{m}^{0}+ix_{m}^{1})=\sum i^{\sum_{j=1}^{m}{\epsilon}_{j}}
L(x_{1}^{{\epsilon}_{1}},\ldots ,x_{m}^{{\epsilon}_{m}}),
\]
where $x_{k}^{0}, x_{k}^{1}$ are vectors in $X$, and the summation is extended over the $2^m$ independent choices of $\epsilon_{k}=0, 1\ (1\leq k\leq m)$. The norm of $\widetilde{L}$ depends on the norm used on $\widetilde{X}$, but continuity is always assured.

In the context of polynomials (see also \cite[p.313]{Taylor}), any $P\in\mathcal{P}(^{m}X; \mathbb{R})$ has a unique complex extension $\widetilde{P}\in\mathcal{P}(^{m}\widetilde{X}; \mathbb{C})$, given by the formula
\begin{equation*}
\widetilde{P}(x+iy)=\sum_{k=0}^{[\frac{m}{2}]}(-1)^{k}{\binom {m}{2k}}L(x^{m-2k}y^{2k})
+i\sum_{k=0}^{[\frac{m-1}{2}]}(-1)^{k}{\binom {m}{2k+1}}L(x^{m-(2k+1)}y^{2k+1})
\end{equation*}
for $x, y$ in $X$, where $P:=\Lh$ for some $L\in\mathcal{L}^{s}(^{m}X; \mathbb{R})$. Here also $\widetilde{P}=\widehat{\widetilde{L}}$.

If $\widetilde{X}$ is the complexification of a real Banach space $X$, each $L\in\mathcal{L}^{s}(^{m}X; \mathbb{R})$ has a unique complex extension $\widetilde{L}\in\mathcal{L}^{s}(^{m}\widetilde{X}; \mathbb{C})$ with $\|L\|\leq \|\widetilde{L}\|$ and $\|P\|\leq \|\widetilde{P}\|$, where $P=\Lh$. We also have \cite[proposition 18]{MST}
\begin{equation}\label{realvscomplex}
\|\widetilde{P}\|\leq 2^{m-1}\|P\|\quad\text{and}\quad\|\widetilde{L}\|\leq 2^{m-1}\|L\|\,.
\end{equation}

\begin{proposition}\label{realthm1}
Let $k_1, \ldots, k_n$ be natural numbers whose sum is $m$. Then for any {\it real} Banach space $X$
\begin{equation}\label{realthm1-1}
\mathbb{R}(k_1, \ldots, k_n; X)\leq 2^{m-1}\mathbb{C}(k_1, \ldots, k_n; \widetilde{X})\leq 2^{m-1}\frac{m^m}{k_{1}^{k_1}\cdots k_{n}^{k_n}}\,\frac{k_{1}!\cdots k_{n}!}{m!}\,,
\end{equation}
where $\widetilde{X}$ is the complexification of the real Banach space $X$.
\end{proposition}

\begin{proof}
Let $\varepsilon>0$. Choose $L\in\mathcal{L}^{s}(^{m} X; \mathbb{R})$ and $x_{i}\in X$, $\|x_{i}\|\leq 1$, $i=1, \ldots, n$, such that
\[
|L(x_{1}^{k_1} \ldots x_{n}^{k_n})|\geq(\mathbb{R}(k_1, \ldots, k_n; X)-\varepsilon) \|\Lh\|\,.
\]
But $L$ has a unique complex extension $\widetilde{L}\in\mathcal{L}^{s}(^{m}\widetilde{X}; \mathbb{C})$ with $\|L\|\leq \|\widetilde{L}\|$ and $\|P\|\leq \|\widetilde{P}\|$, where $P=\Lh$. If we use the first inequality in \eqref{realvscomplex}, it follows that
\[
|\widetilde{L}(x_{1}^{k_1} \ldots x_{n}^{k_n})|\geq\frac{\mathbb{R}(k_1, \ldots, k_n; X)-\varepsilon} {2^{m-1}}\|\widetilde{P}\|\,.
\]
Hence,
\[
\mathbb{R}(k_1, \ldots, k_n; X)-\varepsilon\leq 2^{m-1}\mathbb{C}(k_1, \ldots, k_n; \widetilde{X})
\]
and from Corollary \ref{Mthmcor1} the proof of \eqref{realthm1-1} follows.
\end{proof}

As an application of Proposition \ref{realthm1} in the case where $k_1= \cdots =k_n=1$, we have $n=m$ and from \eqref{realthm1-1} we get
\begin{equation}\label{realthm1-2}
\mathbb{R}(m; X)\leq 2^{m-1} \frac{m^m}{m!}\,.
\end{equation}
However, the last estimate is not the optimal one, since by \eqref{polarrange} we know that $\mathbb{R}(m; X)\leq m^{m}/m!$; i.e., that constant $m^{m}/m!$ is the best possible. A sharper estimate for $\mathbb{R}(m; X)$ is provided later on; see Propositions \ref{realthm2} and \ref{realthm3}. Moreover, Proposition \ref{realthm1}, together with Proposition \ref{Cpolarconst} shall be used to proving the following result regarding the asymptotic growth of the polarization constant $\mathbb{R}(k_1, \ldots, k_n; X)$.

\begin{corollary}{\label{correalthm1}}
Let $k_1, \ldots, k_n$ be natural numbers whose sum is $m$. For any real Banach space $X$ and for $n$ fixed
\[
\limsup_{m\to \infty} \mathbb{R}(k_1, \ldots, k_n; X)^{1/m}\leq 2\,.
\]
\end{corollary}

We also refer to proposition $2.7$ in \cite{DGR} for an analogous result.

Next, we improve the estimate in \eqref{realthm1-1} by using two different techniques.

\subsection{\bf Using the polarization formula.}

First we need a lemma which is an application of the polarization formula \eqref{polarformula}.

\begin{lemma} \label{gpolarformula}
Let $X$ be a Banach space and let $L\in\mathcal{L}^{s}(^{m} X)$. Then for any $x_{1},\ldots,x_{n}\in X$
\[
L(x_{1}^{k_1}\ldots x_{n}^{k_n})=\frac{1}{m!}\int_{0}^{1} r_{1}(t)\cdots r_{m}(t) \Lh\left(\sum_{i=1}^{k_1} r_{i}(t)x_{1}+ \cdots+ \sum_{i=m-k_{n}+1}^{m} r_{i}(t)x_{n}\right)\, dt\,,
\]
where $k_1, \ldots, k_n$ are nonnegative integers whose sum is $m$.
\end{lemma}

\begin{proposition}\label{realthm2}
Let $k_1, \ldots, k_n$ be natural numbers whose sum is $m$, $m\geq 2$. Then for any real Banach space $X$
\begin{equation}\label{realthm2-1}
\mathbb{R}(k_1, \ldots, k_n; X)\leq\frac{n^{m-1}}{m!}\big(k_{1}^{m-1}+\cdots +k_{n}^{m-1}\big)\,.
\end{equation}
\end{proposition}

\begin{proof}
Let $L\in\mathcal{L}^{s}(^{m} X; \mathbb{R})$ and let $x_{1},\ldots,x_{n}$ be unit vectors in $X$. From Lemma \ref{gpolarformula} and H\"{o}lder's inequality we have
\begin{eqnarray*}
|L(x_{1}^{k_1}\ldots x_{n}^{k_n})| &\leq& \frac{1}{m!}\|\Lh\|\int_{0}^{1} \left\{\bigg|\sum_{i=1}^{k_1} r_{i}(t)\bigg|+ \cdots +\bigg|\sum_{i=m-k_{n}+1}^{m} r_{i}(t)\bigg|\right\}^{m}\, dt\\
&\leq& \frac{n^{m-1}}{m!}\|\Lh\|\left\{\int_{0}^{1} \bigg|\sum_{i=1}^{k_1} r_{i}(t)\bigg|^{m}\,dt+ \cdots +\int_{0}^{1} \bigg|\sum_{i=m-k_{n}+1}^{m} r_{i}(t)\bigg|^{m}\, dt\right\}\,.
\end{eqnarray*}
The below inequality is an application of the Riesz-Thorin interpolation theorem, see \cite{BL}(we also refer to inequality ($27$) in \cite[theorem 5]{WW}). For $j<k$, and for $m'$ being the conjugate exponent of $m$, $m\geq 2$, we have:
\[
\int_{0}^{1} \bigg|\sum_{i=j+1}^{k} r_{i}(t)\bigg|^{m}\, dt\leq (k-j)^{m/m'}=(k-j)^{m-1}\,.
\]
Hence,
\begin{eqnarray*}
|L(x_{1}^{k_1}\ldots x_{n}^{k_n})|\leq\frac{n^{m-1}}{m!}\big(k_{1}^{m-1}+\cdots +k_{n}^{m-1}\big) \|\Lh\|
\end{eqnarray*}
and the proof of \eqref{realthm2-1} follows.
\end{proof}

Inequality \eqref{realthm2-1} recovers the optimal estimate $m^m/m!$; indeed, in the special case
$k_1= \cdots =k_n=1$, inequality  \eqref{realthm2-1} gives
\begin{equation}\label{realthm2-2}
\mathbb{R}(m; X)\leq \frac{m^m}{m!}\,.
\end{equation}

Consider now the case where $L\in\mathcal{L}^{s}(^{m} L^{p}(\mu); \mathbb{R})$, $1\leq p\leq\infty$. Let $k_1, \ldots, k_n$ be natural numbers whose sum is $m$, $m\geq 2$. If $x_1, \ldots, x_n$ are norm-one vectors in $L^{p}(\mu)$ with \textit{disjoint supports}, using Lemma \ref{gpolarformula} and Clarkson inequalities the following estimates were derived in theorem $3.6$ in \cite{PSar}
\[
|L(x_{1}^{k_1}\ldots x_{n}^{k_n})|\leq\mathbb{R}(k_{1}, \ldots,k_{n}; L^{p}(\mu))\|\Lh\|\,,
\]
where
\[
\mathbb{R}(k_{1}, \ldots,k_{n}; L^{p}(\mu))=\begin{cases} \frac{1}{m!}\left(k_{1}^{p-1}+\cdots+k_{n}^{p-1}\right)^{m/p} &\text{if $p\geq m$}\\
\\
\frac{n^{(m-p)/p}}{m!}\left(k_{1}^{m-1}+\cdots+k_{n}^{m-1}\right) &\text{if $1\leq p\leq m$\,.}
                           \end{cases}
\]
It is easy to check that these two last estimates are smaller than the estimate given in \eqref{realthm2-1} (in the first estimate just use H\"{o}lder's inequality).

For some more estimates on $L^{p}(\mu)$ spaces see theorems $3.2$ and $3.12$ in \cite{PSar}.

\subsection{\bf Using the Hilbert space case.}

Let $F\in\mathcal{L}^{s}(^{m} H)$, where $H$ is a Hilbert space. Recall from the introduction that $\mathbb{K}(k_1, \ldots, k_n; H)=1$, where $k_1, \ldots, k_n$ are nonnegative integers whose sum is $m$. So, if $y_1, \ldots, y_n$ are unit vectors in $H$, then $|F(y_{1}^{k_1}\ldots y_{n}^{k_n})|\leq \|\Fh\|$.

\begin{proposition}\label{realthm3}
Let $k_1, \ldots, k_n$ be nonnegative integers whose sum is $m$. Then for any {\it real} Banach space $X$
\begin{equation}\label{realthm3-1}
\mathbb{R}(k_1, \ldots, k_n; X)\leq\sqrt{\frac{m^m}{k_{1}^{k_1}\cdots k_{n}^{k_n}}}\,.
\end{equation}
\end{proposition}

\begin{proof}
Let $L\in\mathcal{L}^{s}(^{m} X; \mathbb{R})$. If $x_{1},\ldots, x_{n}$ are unit vectors in $X$, put $r_j=(k_j/m)^{1/2}$ for $j=1, \ldots, n$. For any $t=(t_{1}, \ldots ,t_{n})$ in the $n$-dimensional Euclidean space $\ell_{n}^{2}$, define
\[
\Fh(t):= \Lh(r_{1}x_{1}t_{1}+\cdots +r_{n}x_{n}t_{n})\,.
\]
Then $\Fh$ is an $m$-homogeneous polynomial on $\ell_{n}^{2}$ with $\|\Fh\|\leq \|\Lh\|$. To see this, observe that for any unit vector $t=(t_{1}, \ldots ,t_{n})$ in $\ell_{n}^{2}$
\[
|\Fh(t)|\leq \|\Lh\|\|r_{1}x_{1}t_{1}+\cdots +r_{n}x_{n}t_{n}\|^{m}\leq\|\Lh\|(r_{1}^{2}+\cdots +r_{n}^{2})^{m/2}(t_{1}^{2}+\cdots +t_{n}^{2})^{m/2}=\|\Lh\|\,.
\]
On the other hand, for any $t=t_{1}e_{1}+\cdots + t_{n}e_{n}$, where $(e_i)_{i=1}^n$ is the standard unit vector basis on the Euclidean space $\mathbb{R}^{n}$, from the multinomial formula
\begin{eqnarray*}
\Fh(t) &=& \sum_{j_1+ \cdots +j_n=m} \frac{m!}{j_{1}! \cdots j_{n}!}F(e_{1}^{j_1}\ldots e_{n}^{j_n}) t_{1}^{j_1}\cdots t_{n}^{j_n}\\
&=& \sum_{j_1+ \cdots +j_n=m} \frac{m!}{j_{1}! \cdots j_{n}!}r_{1}^{j_1}\cdots r_{n}^{j_n} L(x_{1}^{j_1}\ldots x_{n}^{j_n}) t_{1}^{j_1}\cdots t_{n}^{j_n}\,.
\end{eqnarray*}
Then by taking partial derivatives
\[
\frac{\partial^{k_2 +\cdots+ k_n}}{\partial{t_{2}^{k_2}}\ldots \partial{t_{n}^{k_n}}} \Fh(e_1)
=\frac{m!}{k_{1}!}F(e_{1}^{k_1}\ldots e_{n}^{k_n})=\frac{m!}{k_{1}!}r_{1}^{k_1}\cdots r_{n}^{k_n} L(x_{1}^{k_1}\ldots x_{n}^{k_n})\,.
\]
So,
\begin{eqnarray*}
|L(x_{1}^{k_1}\ldots x_{n}^{k_n})| &=& \frac{1}{r_{1}^{k_1}\cdots r_{n}^{k_n}}|F(e_{1}^{k_1}\ldots e_{n}^{k_n})|\\
&\leq& \sqrt{\frac{m^m}{k_{1}^{k_1}\cdots k_{n}^{k_n}}}\,\|\Fh\|
\leq\sqrt{\frac{m^m}{k_{1}^{k_1}\cdots k_{n}^{k_n}}}\,\|\Lh\|
\end{eqnarray*}
and the proof of \eqref{realthm3-1} follows.
\end{proof}

By the use of a different technique, Proposition \ref{realthm3} was also proved by Harris in \cite[corollary 7]{H2}. The case $n=2$ of the above proposition is in fact part (a) of the corollary in \cite{Sar5}. However, still the estimate \eqref{realthm3-1} is far from being optimal; indeed, consider the case where $k_{1}=\cdots =k_{n}=1$, then the right-hand side of \eqref{realthm3-1} becomes $m^{m/2}$, and clearly we have $m^{m}/m!<m^{m/2}$ for $m\geq 3$.

The following is an immediate consequence of Proposition \ref{realthm2} and Proposition \ref{realthm3}.

\begin{corollary}\label{corsummary}
Let $k_1, \ldots, k_n$ be natural numbers whose sum is $m$. Then for any {\it real} Banach space $X$
\begin{equation}\label{corsummary-1}
\mathbb{R}(k_1, \ldots, k_n; X)\leq \min\left\{\frac{n^{m-1}}{m!}\big(k_{1}^{m-1}+\cdots +k_{n}^{m-1}\big)\,,\,\, \sqrt{\frac{m^m}{k_{1}^{k_1}\cdots k_{n}^{k_n}}}\right\}\,.
\end{equation}
\end{corollary}

In particular, from \eqref{corsummary-1} we have $\mathbb{R}(2, 2; X)\leq 4$. But, $\mathbb{R}(2, 2; X)\leq 3$ and $3$ is best possible. To see this, consider $L\in\mathcal{L}^{s}(^{4} X; \mathbb{R})$. For $x, y\in X$
\[
\Lh(x)+ \Lh(y)+ 6L(x^{2}y^{2})=\int_{0}^{1} \Lh(r_{1}(t)x+ r_{2}(t)y)\, dt
\]
where $r_{1}$, $r_{2}$ are the first two Rademacher functions. Thus
\[
|L(x^{2}y^{2})|\leq 3\|\Lh\|
\]
for all unit vectors $x$ and $y$. The following example, see example I.2.16 in \cite{Sar1} or example $1$ in \cite{Sar5}, shows that $3$ is the best constant and therefore
\[
\mathbb{R}(2, 2; X)=3>\frac{8}{3}\geq\mathbb{C}(2, 2; X)\,.
\]

\begin{example}\label{real-2-2}
Consider the {\it real} space $\ell_{4}^{\infty}$. For $x=(x_{1}, x_{2}, x_{3}, x_{4})\in \ell_{4}^{\infty}$, choose the $4$-homogeneous polynomial
\[
\Lh(x)=(x_{1}^{2}- x_{2}^{2})^{2}-  (x_{3}^{2}- x_{4}^{2})^{2}\,.
\]
It is easy to verify that $\|\Lh\|=1$. For $x=(x_{1}, x_{2}, x_{3}, x_{4})$ and $y=(y_{1}, y_{2}, y_{3}, y_{4})$ in $\ell_{4}^{\infty}$ we can easily check that
\begin{eqnarray*}
L(x^{2}y^{2}) &=& x_{1}^{2}y_{1}^{2}+x_{2}^{2}y_{2}^{2}-x_{3}^{2}y_{3}^{2}-x_{4}^{2}y_{4}^{2}
-\frac{1}{3}\big(x_{1}^{2}y_{2}^{2}+4x_{1}x_{2}y_{1}y_{2}+x_{2}^{2}y_{1}^{2}\big)\\
& & +\frac{1}{3}\big(x_{3}^{2}y_{4}^{2}+4x_{3}x_{4}y_{3}y_{4}+x_{4}^{2}y_{3}^{2}\big)\,.
\end{eqnarray*}
Taking $x=(1, 1, 0, 1)$ and $y=(1, -1, 1, 0)$ we get
\[
|L(x^{2}y^{2})|= 3\|\Lh\|\,.
\]
\end{example}

The problem of determining the best constant in \eqref{corsummary-1} is open. Harris in his commentary to problems $73$ and $74$ of Mazur and Orlicz in the Scottish Book \cite{Scottish} discusses in detail this problem.

Suppose $M_{m, k}$ and $K_{m, k}$, $k\leq m$, are the smallest numbers with the property
\[
\|\Dh^{k}\Lh\|\leq M_{m, k}\|\Lh\|\quad\text{and}\quad\|D^{k}\Lh\|\leq K_{m, k}\|\Lh\|\,,
\]
respectively, for every continuous $m$-homogeneous polynomial $\Lh$ on any \textit{real} Banach space $X$. Then, from Example \ref{Mexample} and the estimate in \eqref{realthm3-1} it follows that
\begin{equation}\label{real-Markov-1}
\frac{m^{m}k!}{(m-k)^{(m-k)}k^{k}}\leq M_{m, k}\leq {m\choose k}\frac{k!m^{m/2}}{(m-k)^{(m-k)/2}k^{k/2}}
\end{equation}
and
\begin{equation}\label{real-Markov-2}
\frac{m^{m}}{(m-k)^{(m-k)}}\leq K_{m, k}\leq {m\choose k}\frac{m^{m/2}k^{k/2}}{(m-k)^{(m-k)/2}}\,.
\end{equation}

The last two inequalities were also proved in \cite{Sar5}. The upper bounds for $M_{m, k}$ and
$K_{m, k}$ are not best possible. From inequality \eqref{real-Markov-1} the upper bound for $M_{4, 2}$ is $48$. But, from Example \ref{real-2-2} we have $M_{4, 2}=36$ and this is the best constant.

\subsection{\bf Markov-type inequalities for polynomials on \textit{real} Banach spaces.}

This subsection is reproduced from \cite{ChSar}. V. A. Markov (brother of A. A. Markov) considered the problem of determining exact bounds for the $k$th derivative of an algebraic polynomial. For $1\leq k\leq m$, if $p\in\mathcal{P}_m(\mathbb{R})$
and $\|p\|_{[-1,1]}\leq 1$, V. A. Markov \cite{M} has shown that
\[
\|p^{(k)}\|_{[-1,1]}\leq T_{m}^{(k)}(1)=\frac{m^2(m^2-1^2)\cdots (m^2-(k-1)^2)}{1\cdot3\cdots (2k-1)}\,.
\]
Recall thst the $m$th Chebyshev polynomial $T_{m}(t)$ is the polynomial agreeing $\cos(m\arccos{t})$ in the range $-1<t<1$.

Let $K\subset\mathbb{R}^{n}$ be a convex body, i.e. a convex compact set with non-empty interior. If $u$ is a unit vector in $\mathbb{R}^{n}$ then there are precisely two support hyperplanes to $K$ having $u$ for a normal vector. The distance $w(u)$ between these parallel support hyperplanes is the width of $K$ in the direction of $u$. The \textit{minimal width} of $K$ is $w(K):= \min_{\|u\|_{2}=1} w(u)$

Consider now the case where $K\subset\mathbb{R}^{n}$ is a centrally symmetric convex body with center at the origin, in other words $K$ is invariant under $x\mapsto -x$. We call $K$ a ball. A ball $K$ is the unit ball of a unique Banach norm $\|\cdot\|_K$ defined by
\[
\|x\|_{K}=\inf\{t>0: \ x/t\in K\},\quad x\in\mathbb{R}^{n}\,.
\]
If $P\in\mathcal{P}_{m}(\mathbb{R}^n)$, $x\in\interior{K}$ and $y\in S_{\mathbb{R}^n}$, the next sharp Bernstein and Markov-type inequalities follow from the work of Sarantopoulos\cite{Sar5}:
\begin{align}
|DP(x)y| &\leq \frac{2m}{w(K)\sqrt{1-\|x\|_K^2}}\|P\|_{K}\,, \label{sar5-1}\\
\|\nabla P\|_K &\leq\frac{2m^2}{w(K)}\|P\|_{K}\,. \label{sar5-2}
\end{align}
In fact, if $X$ is a \textit{real} Banach space and $P\in\mathcal{P}_{m}(X)$, $\|P\|\leq 1$, for the first Fr\'echet derivative of $P$ it was proved in \cite{Sar5} that
\begin{equation}\label{firstgen}
\|DP(x)\|\leq\min\left\{m\frac{\sqrt{1-P(x)^2}}{\sqrt{1-\|x\|^2}},\, m^{2}\right\},\quad\text{for every $\|x\|<1$}\,.
\end{equation}
Using methods of several complex variables, inequalities (\ref{sar5-1}) and (\ref{sar5-2})  were proved independently by Baran \cite{BARAN}.

Finally, the proof of Markov's inequality for any derivative of a polynomial on a \textit{real} Banach space $X$ was given my Skalyga in \cite{Sk1}, see also \cite{Sk2}. In 2010 Harris \cite{H4} gave another proof which depends on a Lagrange interpolation formula for the Chebyshev nodes and a Christoffel-Darboux identity for the corresponding bivariate Lagrange polynomials \cite{H3}.

\begin{theorem}$(${\bf V. A. Markov's theorem}$)$\cite{Sk1, Sk2, H4}\label{Sk1-1}
Let $X$ be a \textit{real} Banach space and let $P\in\mathcal{P}_{m}(X)$ with $\|P\|\leq 1$. Then for any $1\leq k\leq m$,
\begin{equation}\label{Mineq}
\|\widehat{D}^{k}P\|\leq T_{m}^{(k)}(1)=\frac{m^2(m^2-1^2)\cdots (m^2-(k-1)^2)}{1\cdot3\cdots (2k-1)}\,.
\end{equation}
\end{theorem}

So, for any continuous polynomial $P$ of degree $m$ on a real Banach space $X$, $T_{m}^{(k)}(1)$ is the best possible constant for $\|\widehat{D}^{k}P(x)\|$, $\|x\|\leq 1$. As mentioned earlier the upper bound for $\|\widehat{D}^{k}P(x)\|$, $\|x\|\leq 1$, as follows by \eqref{real-Markov-1} is not the best possible. However, it improves the one given in \eqref{Mineq} for the case where $P$ is an \textit{$m$-homogeneous} polynomial on $X$.

\end{document}